\documentclass[12pt]{article}

\usepackage{mathrsfs}
\usepackage{amsmath, epsfig, cite}
\usepackage{amssymb}
\usepackage{amsfonts}
\usepackage{latexsym}
\usepackage{amsthm}

\makeatletter
\renewcommand{\@seccntformat}[1]{{\csname the#1\endcsname}.\hspace{.5em}}
\makeatother

\newtheorem{thm}{Theorem}[section]
\newtheorem{prop}[thm]{Proposition}

\newtheorem{cor}[thm]{Corollary}

\newtheorem{remark}[thm]{Remark}

\newcommand{\pf}{\noindent{\it Proof.} }
\def\N{\mathbb{N}}

\def\des{{\rm des}}
\def\maj{{\rm maj}}

\renewcommand{\qed}{\hfill$\Box$\medskip}

\numberwithin{equation}{section}

\begin{document}

\renewcommand{\thefootnote}{*}

\begin{center}
{\Large\bf 
Basic and bibasic identities related to\\[5pt] divisor functions}
\end{center}

\vskip 2mm \centerline{Victor J. W. Guo$^1$  and Jiang Zeng$^{2}$}
\begin{center}
{\footnotesize $^1$School of Mathematical Sciences, Huaiyin Normal University, Huaian, Jiangsu 223300,
 People's Republic of China\\
{\tt jwguo1977@aliyun.com}\\[10pt]
$^2$Universit\'e de Lyon; Universit\'e Lyon 1; Institut Camille
Jordan, UMR 5208 du CNRS;\\ 43, boulevard du 11 novembre 1918,
F-69622 Villeurbanne Cedex, France\\
{\tt zeng@math.univ-lyon1.fr } }
\end{center}

\vskip 0.7cm \noindent{\small{\bf Abstract.}}  Using basic hypergeometric functions and partial fraction decomposition we
give  a new kind of generalization of identities due to Uchimura, Dilcher, Van Hamme,
Prodinger, and Chen-Fu related to divisor functions.
An identity relating Lambert series to Eulerian polynomials is also proved.




\section{Introduction}

The divisor function $\sigma_m(n)$ for a natural number $n$ is defined as the
sum of the $m$th powers of the (positive) divisors of $n$, i.e.,
$\sigma_m(n)=\sum_{d|n}d^m$.
Throughout this paper, we assume that $|q|<1$. The generating function of $\sigma_m(n)$ has
an explicit  Lambert series expansion (see \cite{Apostol})
\begin{align}
\sum_{n=1}^\infty \sigma_m(n)q^n=\sum_{n=1}^\infty \frac{n^m q^n}{1-q^n}. \label{eq:divisor}
\end{align}

In 1981, Uchimura \cite{Uchimura81} rediscovered  an  identity due to
 Kluyver \cite{Kluyver}  (see also Dilcher \cite{Dilcher}):
\begin{align}
\sum_{k=1}^{\infty}(-1)^{k-1}\frac{q^{k+1\choose 2}}{(q;q)_k(1-q^k)}
=\sum_{k=1}^\infty \frac{q^k}{1-q^k},  \label{eq:u81}
\end{align}
where $(a;q)_n=(1-a)(1-aq)\cdots(1-aq^{n-1})$ for $n\geqslant 0$.  Since then,
many authors have given different generalizations of \eqref{eq:u81}
(see \cite{Ando,AU, Dilcher, FL, FL2, GP, GZh, Hoffman, Hamme, IS, Prodinger, Prodinger2, Uchimura81, Uchimura84, Zeng}).
For example, Van Hamme \cite{Hamme} gave a finite form of \eqref{eq:u81} as follows:
\begin{align}
\sum_{k=1}^{n}(-1)^{k-1}{n\brack k}_q\frac{q^{k+1\choose 2}}{1-q^k}
=\sum_{k=1}^n \frac{q^k}{1-q^k}, \label{eq:hamme}
\end{align}
where the $q$-binomial coefficients ${n\brack k}_q$  are defined by
$$
{n\brack k}_q=\frac{(q;q)_n}{(q;q)_k (q;q)_{n-k}}.
$$
Uchimura \cite{Uchimura84} proved that
\begin{align}
\sum_{k=1}^{n}(-1)^{k-1}{n\brack k}_q\frac{q^{k+1\choose 2}}{1-q^{k+m}}
=\sum_{k=1}^n \frac{q^k}{1-q^k}{k+m\brack m}_q^{-1},\quad m\geqslant 0.  \label{eq:uch}
\end{align}
Dilcher \cite{Dilcher} established the following multiple series generalization of \eqref{eq:rdiv}:
\begin{align}
\sum_{k=1}^{n}(-1)^{k-1}{n\brack k}_q\frac{q^{{k\choose 2}+km}}{(1-q^k)^m}
=\sum_{1\leqslant k_1\leqslant\cdots\leqslant k_m\leqslant n} \frac{q^{k_1+\cdots+k_m}}{(1-q^{k_1})\cdots(1-q^{k_m})}.  \label{eq:dilch}
\end{align}
Prodinger \cite{Prodinger} proved that
\begin{align}
\sum_{\substack{k=0\\ k\neq m}}^{n}(-1)^{k-1}{n\brack k}_q\frac{q^{k+1\choose 2}}{1-q^{k-m}}
=(-1)^m q^{m+1\choose 2}{n\brack m}_q\sum_{\substack{k=0\\ k\neq m}}^{n} \frac{q^{k-m}}{1-q^{k-m}},\quad 0\leqslant m\leqslant n.
\label{eq:prodinger}
\end{align}
 Using partial fraction decomposition the second author \cite[(7)]{Zeng} obtained
 the following common generalization of  Dilcher's identity~\eqref{eq:dilch} and of
some  identities due to Fu and Lascoux~\cite{FL,FL2}:
\begin{align}
 \sum_{k=i}^{n}(-1)^{k-i}{n\brack k}_q{k\brack i}_q\frac{q^{{k-i\choose 2}+km}}{(1-zq^k)^{m}}
=\frac{q^i(q;q)_{i-1}(q;q)_n}{(q;q)_i(zq;q)_n}
h_{m-1}\biggl(\frac{q^i}{1-zq^i}, \ldots,  \frac{q^n}{1-zq^n} \biggr),\label{eq:flz}
\end{align}
where $1\leqslant i\leqslant n$ and $h_k(x_1, \ldots, x_n)$ is the $k$-th homogeneous symmetric polynomial in $x_1, \ldots, x_n$
 defined by
\begin{align*}
h_k(x_1, \ldots, x_n)=\sum_{1\leqslant i_1\leqslant \cdots \leqslant i_k\leqslant n} x_{i_1}\cdots x_{i_k}=\sum_{\alpha_1+\cdots +\alpha_n=k}x_1^{\alpha_1}\cdots x_n^{\alpha_n}.
\end{align*}
We note that Ismail and Stanton \cite[Theorem 2.2]{IS} have   rediscovered  the $i=1$ case of
\eqref{eq:flz} as well as some other results in \cite{Zeng}.

In this paper we shall give a different kind of generalizations of \eqref{eq:uch}--\eqref{eq:prodinger}.
Our  starting point  is an identity of  Chen and Fu \cite[(3.3)]{CF},
which corresponds to  the $(r, x)=(0,1)$  case of the following result.
\begin{thm}\label{thm:01}
Let $m,n\geqslant 0$ and $0\leqslant r\leqslant m$. Then
\begin{align}
\sum_{k=0}^m\frac{(-1)^k p^{{k+1\choose 2}-rk}}{(p;p)_{k}(p;p)_{m-k}(xp^{k};q)_{n+1}}
=x^r\sum_{k=0}^n\frac{(-1)^k q^{{k+1\choose 2}+rk}}{(q;q)_{k}(q;q)_{n-k}(xq^{k};p)_{m+1}}.
\label{eq:new}
\end{align}
\end{thm}
We shall give a proof of Theorem \ref{thm:01} along with its several consequences and variations  in Section~2.

Recently Liu~\cite[Proposition 4.1]{Liu} has obtained the following
formula,  the left-hand side of which specializes to the  Lambert series in \eqref{eq:u81} when $a=1$:
\begin{align}
\sum_{n=1}^\infty \frac{a q^n}{1-a q^n}
=\sum_{n=1}^\infty \frac{(1-aq^{2n})a^n q^{n^2}}{(1-q^n)(1-aq^n)}. \label{eq:liu}
\end{align}
Motivated by the identity \eqref{eq:liu}, we shall generalize the expansion of  the Lambert series \eqref{eq:divisor}
by using Carlitz's $q$-Eulerian polynomials $A_{n}(t;q)$.
Recall (see \cite{Foata}) that  these $q$-Eulerian polynomials $A_{n}(t;q)$ are defined by
\begin{align}
\sum_{j=0}^\infty t^{j} [j+1]_{q}^n
=\frac{A_n(t;q)}{(t;q)_{n+1}},\label{eq:eulerian}
\end{align}
where $[n]_{q}=(1-q^{n})/(1-q)$ for any positive integer $n$. A
well-known combinatorial interpretation for $A_{n}(t;q)$ is given by
$$
A_n(t;q)=\sum_{\sigma\in S_{n}}t^{\des\,\sigma}q^{\maj\,\sigma},
$$
where  $S_{n}$ is the set of permutations of $\{1,\ldots,n\}$ and
$$
\des\,\sigma=\sum_{{\sigma(i)>\sigma(i+1})}1,\quad \maj\,\sigma=\sum_{\sigma(i)>\sigma(i+1)}i.
$$
The first values of the polynomials $A_{n}(t;q)$ are the following:
$$
A_{0}(t;q)=A_{1}(t;q)=1;\; A_{2}(t;q)=1+tq;\; A_{3}(t;q)=1+2tq(q+1)+t^{2}q^{3}.
$$
We generalize Liu's result~\eqref{eq:liu} as follows.
\begin{thm}\label{thm:main}
For $m\geqslant 0$ and $|a|\leqslant 1$,  we have
\begin{align}
\sum_{n=1}^\infty \frac{a [n]_{p}^m q^n}{1-a q^n}
=\sum_{n=1}^\infty \frac{[n]_{p}^m(1-aq^{2n})a^n q^{n^2}}{(1-q^n)(1-aq^n)}
+\sum_{k=1}^m {m\choose k}\sum_{n=1}^\infty \frac{[n]_{p}^{m-k }a^np^{kn} q^{n^2+n}A_k(q^n;p)}{(q^n;p)_{k+1}}.\label{eq:main1}
\end{align}
\end{thm}
When $p=1$ the polynomials $A_n(t):=A_{n}(t;1)$ are
the Eulerian polynomials (see \cite{Foata, Stanley}).
For example, we have
\begin{align*}
A_1(t)=1,\ A_2(t)=1+t,\ A_3(t)=1+4t+t^2,\ A_4(t)=1+11t+11t^2+t^3.
\end{align*}
Clearly, the above theorem reduces to  the following result for $p=1$.
\begin{cor}\label{thm:main2}
For $m\geqslant 0$ and $|a|\leqslant 1$, we have
\begin{align}
\sum_{n=1}^\infty \frac{a n^m q^n}{1-a q^n}
=\sum_{n=1}^\infty \frac{n^m(1-aq^{2n})a^n q^{n^2}}{(1-q^n)(1-aq^n)}
+\sum_{k=1}^m {m\choose k}\sum_{n=1}^\infty \frac{n^{m-k }a^n q^{n^2+n}A_k(q^n)}{(1-q^n)^{k+1}}.\label{eq:main2}
\end{align}
\end{cor}
When $m=0$, the above theorem reduces to Liu's formula~\eqref{eq:liu}.
However, our proof for \eqref{eq:main1} is
quite different from Liu's proof for \eqref{eq:liu}.

Dividing  both sides of \eqref{eq:main2} by $a$ and setting $a=0$ we obtain
\begin{align*}
\sum_{n=1}^\infty  n^m q^n
=\frac{q}{1-q}
+\sum_{k=1}^m {m\choose k}\frac{q^{2}A_k(q)}{(1-q)^{k+1}}.
\end{align*}
This is equivalent to the $q=1$ case of \eqref{eq:eulerian} by using the symmetry of the Eulerian polynomials
$t^{n-1}A_n(t^{-1})=A_n(t)$ ($n\geqslant 1$) and the recurrence relation
\begin{align*}
A_0(t)=1,\quad A_n(t)=\sum_{k=0}^{n-1}{n\choose k} (t-1)^{n-k-1} A_k(t).
\end{align*}

In the next section we shall first give a simple proof of Theorem~\ref{thm:01}
using the $q$-series theory
and  then derive some consequences and variations.
We will prove Theorem~\ref{thm:main} in Section~3 and  give
a finite analogue of Uchimura's variation when $a=1$ along with
a connection between partitions and the number of divisors.
Finally we give new generalizations
of \eqref{eq:uch}--\eqref{eq:prodinger} in Section~4.
\section{Proof of Theorem~\ref{thm:01} and some consequences}
\subsection{Proof of Theorem~\ref{thm:01}}
Applying the $q$-binomial theorem \cite[(3.3.6)--(3.3.7)]{Andrews}:
\begin{align}
(x;q)_N &=\sum_{j=0}^N{N\brack j}_{q} (-1)^j  q^{{j\choose 2}}x^j, \label{eq:euler1}\\
\frac{1}{(x;q)_N} &=\sum_{j=0}^\infty{N+j-1\brack j}_q x^j, \label{eq:euler2}
\end{align}
we have
\begin{align}
\sum_{k=0}^m\frac{(-1)^k p^{{k+1\choose 2}-rk}}{(p;p)_{k}(p;p)_{m-k}(p^{k}x;q)_{n+1}}
&=\sum_{k=0}^m \frac{ (-1)^k p^{{k+1\choose 2}-rk}}{(p;p)_{k}(p;p)_{m-k}}
\sum_{j=0}^\infty {n+j\brack j}_q (p^{k}x)^j \nonumber \\
&=\sum_{j=0}^\infty {n+j\brack j}_q \frac{x^j}{(p;p)_m}
\sum_{k=0}^m {m\brack k}_p (-1)^k p^{{k\choose 2}+(j-r+1)k}  \nonumber \\
&=\frac{1}{(p;p)_m}\sum_{j=0}^\infty {n+j\brack j}_q x^j (p^{j-r+1};p)_m \nonumber \\
&=\frac{1}{(p;p)_m}\sum_{j=0}^\infty {n+r+j\brack r+j}_q x^{r+j} (p^{j+1};p)_m  \nonumber \\
&=\sum_{j=0}^\infty x^{r+j}{m+j\brack m}_p {n+r+j\brack n}_q.  \label{eq:newpf}
\end{align}
Similarly, we have
\begin{align*}
\sum_{k=0}^n\frac{(-1)^k q^{{k+1\choose 2}+rk}x^r}{(q;q)_{k}(q;q)_{n-k}(q^{k}x;p)_{m+1}}
&=\frac{x^r}{(q;q)_n}\sum_{j=0}^\infty {m+j\brack j}_p x^j (q^{j+r+1};q)_n,
\end{align*}
which is equal to the right-hand side of \eqref{eq:newpf}. This completes the proof.
\qed

\subsection{An infinite version}
Letting $r=0$ in \eqref{eq:new},
we obtain the following symmetric bibasic transformation formula.
\begin{cor}For $m,n\geqslant 0$, there holds
\begin{align}
\sum_{k=0}^m\frac{(-1)^k p^{{k+1\choose 2}}}{(p;p)_{k}(p;p)_{m-k}(xp^{k};q)_{n+1}}
=\sum_{k=0}^n\frac{(-1)^k q^{{k+1\choose 2}}}{(q;q)_{k}(q;q)_{n-k}(xq^{k};p)_{m+1}}.
\label{eq:new2}
\end{align}
\end{cor}
In fact, we have the following infinite form of \eqref{eq:new2}.

\begin{thm}\label{thm:02}
Let $|a|,|b|,|p|,|q|<1$. Then
\begin{align}
\frac{(a;p)_\infty}{(p;p)_\infty}\sum_{k=0}^\infty
\frac{(p/a;p)_k (xbp^k;q)_\infty}{(p;p)_k (xp^k;q)_{\infty}} a^k
=\frac{(b;q)_\infty}{(q;q)_\infty}\sum_{k=0}^\infty
\frac{(q/b;q)_k (xaq^k;p)_\infty}{(q;q)_k (xq^k;p)_{\infty}} b^k,  \label{eq:sym}
\end{align}
where $(a;q)_\infty=\lim_{n\to\infty}(a;q)_n$.
\end{thm}

Note that \eqref{eq:sym} may be deemed a bibasic extension of the second iteration of the Heine transformation.
(The $p=q$ case can be seen to reduce to a special case of \cite[Appendix (III.2)]{GR}.)

\pf
Using the $q$-binomial theorem~\cite[Theorem 2.1]{Andrews}:
$$
\sum_{n=0}^\infty \frac{(a;q)_n}{(q;q)_n}x^n =\frac{(ax;q)_\infty}{(x;q)_\infty},
$$
we can write the left-hand side of \eqref{eq:sym} as
\begin{align}
\frac{(a;p)_\infty}{(p;p)_\infty}\sum_{k=0}^\infty \frac{(p/a;p)_k}{(p;p)_k}a^k\sum_{j=0}^\infty \frac{(b;q)_j}{(q;q)_j}(p^kx)^j
&=\frac{(a;p)_\infty}{(p;p)_\infty}\sum_{j=0}^\infty \frac{(b;q)_j}{(q;q)_j}x^j \sum_{k=0}^\infty \frac{(p/a;p)_k}{(p;p)_k}(ap^j)^k \nonumber \\
&=\frac{(a;p)_\infty}{(p;p)_\infty}\sum_{j=0}^\infty \frac{(b;q)_j (p^{j+1};p)_\infty}{(q;q)_j (ap^j;p)_\infty}x^j\nonumber\\
&=\sum_{j=0}^\infty \frac{(b;q)_j (a;p)_j}{(q;q)_j (p;p)_j}x^j.\label{eq:symbis}
\end{align}
Therefore, by symmetry, both sides of \eqref{eq:sym} are equal to \eqref{eq:symbis}.  \qed

When $a=p^{m+1}$ and $b=q^{n+1}$, the identity \eqref{eq:sym}
reduces to \eqref{eq:new2}.
Letting $p=q^2$, $r=0$, $x=q^{2}$ and $m,n\to\infty$ in \eqref{eq:new}, we obtain
\begin{cor}There holds
\begin{align*}
&\hskip -3mm
\sum_{k=0}^\infty (-1)^k (q;q^2)_{k+1} q^{k(k+1)}
=\sum_{k=0}^\infty \frac{q^{(2k+1)k}} {(q;q^2)_k}
-\frac{(q^2;q^2)_\infty}{(q;q^2)_\infty}
\sum_{k=0}^\infty\frac{ q^{(2k+1)(k+1)}}{(q^2;q^2)_{k}}.
\end{align*}
\end{cor}

Writing \eqref{eq:new} as
\begin{align}
&\hskip -2mm \sum_{k=1}^m\frac{(-1)^k p^{{k+1\choose 2}-rk}}{(p;p)_{k}(p;p)_{m-k}(p^{k}x;q)_{n+1}}
-\sum_{k=1}^n\frac{x^r (-1)^k q^{{k+1\choose 2}+rk}}{(q;q)_{k}(q;q)_{n-k}(q^{k}x;p)_{m+1}} \nonumber \\
&=\frac{x^r(p;p)_m (xq;q)_{n}-(q;q)_n(xp;p)_{m}}{(p;p)_m (q;q)_n (xp;p)_{m}(xq;q)_{n}(1-x)}, \label{eq:new3}
\end{align}
and letting $x\to 1$ in \eqref{eq:new3} and applying l'H\^opital's rule,
we are led to the following result:

\begin{cor}For $m,n\geqslant 0$ and $0\leqslant r\leqslant m$, there holds
\begin{align}
&\hskip -3mm  \sum_{k=1}^{m}\frac{(-1)^k p^{{k+1\choose 2}-rk}}{(p;p)_k (p;p)_{m-k}(p^k;q)_{n+1}}
-\sum_{k=1}^{n}\frac{(-1)^k q^{{k+1\choose 2}+rk}}{(q;q)_k (q;q)_{n-k}(q^k;p)_{m+1}}\nonumber\\
&=\frac{1}{(p;p)_{m}(q;q)_{n}}
\left(-r+\sum_{k=1}^n\frac{q^k}{1-q^k}-\sum_{k=1}^m\frac{p^k}{1-p^k}\right). \label{cor:newnew}
\end{align}
\end{cor}
Note that, by the symmetry of $m$ and $n$, the identity \eqref{cor:newnew} also holds for $-n\leqslant r\leqslant 0$.
In particular, when $m=n$ and $p=q$, we have
\begin{cor} For $n\geqslant 0$ and $|r|\leqslant n$, there holds
\begin{align*}
\sum_{k=1}^{n}\frac{(-1)^{k-1} q^{{k+1\choose 2}-rk}(1-q^{2rk})}{(q;q)_k (q;q)_{n-k}(q^k;q)_{n+1}}
=\frac{r}{(q;q)_n^2}.
\end{align*}
\end{cor}

Taking $r=0$ in \eqref{cor:newnew}, we obtain the following bibasic generalization of \cite[Corollary 6.4]{GZh}.

\begin{cor}For $m,n\geqslant 0$, there holds
\begin{align}
&\hskip -3mm  \sum_{k=1}^{m}\frac{(-1)^k p^{{k+1\choose 2}}}{(p;p)_k (p;p)_{m-k}(p^k;q)_{n+1}}
-\sum_{k=1}^{n}\frac{(-1)^k q^{{k+1\choose 2}}}{(q;q)_k (q;q)_{n-k}(q^k;p)_{m+1}}\nonumber\\
&=\frac{1}{(p;p)_{m}(q;q)_{n}}
\left(\sum_{k=1}^n\frac{q^k}{1-q^k}-\sum_{k=1}^m\frac{p^k}{1-p^k}\right). \label{cor:new}
\end{align}
\end{cor}

Letting $p=q^2$ and $m=n$ in \eqref{cor:new}, we have
\begin{cor}For $n\geqslant 0$, there holds
\begin{align}
&\hskip -3mm
\sum_{k=1}^{n}\frac{(-1)^k q^{k(k+1)}}{(q^2;q^2)_k (q^2;q^2)_{n-k}(q^{2k};q)_{n+1}}
-\sum_{k=1}^{n}\frac{(-1)^k q^{{k+1\choose 2}}}{(q;q)_k (q;q)_{n-k}(q^k;q^2)_{n+1}} \nonumber \\
&=\frac{1}{(q^2;q^2)_{n}(q;q)_{n}} \sum_{k=1}^n\frac{q^k}{1-q^{2k}}.  \label{eq:long}
\end{align}
\end{cor}

Letting $n$ tend to $\infty$ in \eqref{eq:long}, we obtain
\begin{align*}
&\hskip -3mm
\sum_{k=1}^\infty\frac{(-1)^k (q;q^2)_k q^{k(k+1)}}{1-q^{2k}}
-\sum_{k=1}^\infty\frac{ q^{(2k+1)k}}{(q;q^2)_k (1-q^{2k})}
+\frac{(q^2;q^2)_\infty}{(q;q^2)_\infty}
\sum_{k=0}^\infty\frac{ q^{(2k+1)(k+1)}}{(q^2;q^2)_{k} (1-q^{2k+1})}\\
&=\sum_{k=1}^\infty \frac{q^k}{1-q^{2k}}.
\end{align*}
Note that (see \cite[(3.1)]{FL2}) the coefficient of $q^n$ in
\begin{align*}
\sum_{k=1}^\infty \frac{q^k}{1-q^{2k}}=\sum_{k=1}^\infty \frac{q^{2k-1}}{1-q^{2k-1}}
=q+q^2+2q^3+q^4+2q^5+2q^6+2q^7+q^8+3q^9+\cdots
\end{align*}
counts the number of (positive) odd divisors of $n$.


\section{Proof of Theorem \ref{thm:main} and some variations}

\subsection{Proof of Theorem \ref{thm:main}}
Assume that $|aq|<1$.  As  
 $[n+j]_{p}=[n]_{p}+p^n[j]_{p}$,
we have
\begin{align}
\sum_{n=1}^\infty \frac{a [n]_{p}^m q^n}{1-a q^n}
&=\sum_{n\geqslant 1}[n]_{p}^m \sum_{k\geqslant 1}a^{k} q^{nk}\nonumber\\
&=\sum_{k\geqslant  n\geqslant 1}[n]_{p}^m a^{k} q^{nk}+\sum_{n>k\geqslant 1}[n]_{p}^m a^{k} q^{nk}
\nonumber\\
&=\sum_{n=1}^\infty\sum_{k=0}^\infty [n]_{p}^m a^{n+k}q^{n(n+k)}
+\sum_{n=1}^\infty\sum_{j=1}^\infty [n+j]_{p}^m a^{n}q^{n(n+j)} \nonumber\\
&=\sum_{n=1}^\infty [n]_{p}^m  \sum_{k=0}^\infty a^{n+k}q^{n(n+k)}
+\sum_{n=1}^\infty [n]_{p}^m  \sum_{j=1}^\infty  a^{n}q^{n(n+j)} \nonumber \\
&\ \quad{}+\sum_{k=1}^m {m\choose k}[n]_{p}^{m-k}
\sum_{n=1}^\infty a^n\,p^{nk}\sum_{j=1}^\infty [j]_{p}^k  q^{n(n+j)}.  \label{eq:simple}
\end{align}
Note that
\begin{align}
\sum_{k=0}^\infty  a^{n+k}q^{n(n+k)}
+\sum_{j=1}^\infty a^{n}q^{n(n+j)}
&=\frac{a^n q^{n^2}}{1-aq^n}+\frac{a^n q^{n^2+n}}{1-q^n}\nonumber\\
&=\frac{(1-aq^{2n})a^n q^{n^2}}{(1-q^n)(1-aq^n)}, \label{eq:simple2}
\end{align}
and, by \eqref{eq:eulerian},
\begin{align}
\sum_{j=1}^\infty [j]_{p}^k q^{n(n+j)}
=q^{n^2+n}\sum_{j=1}^\infty [j]_{p}^k q^{n(j-1)}
=\frac{q^{n^2+n} A_k(q^n;p)}{(q^n;p)_{k+1}}. \label{eq:eulerian2}
\end{align}
Combining \eqref{eq:simple}, \eqref{eq:simple2} and \eqref{eq:eulerian2},
we complete the proof of \eqref{eq:main1}.

\begin{prop} For $m\geqslant 0$ and $|a|\leqslant 1$,  we have
\begin{align}
\sum_{n=1}^\infty {n\choose m}\frac{a q^n}{1-a q^n}
=\sum_{n=1}^\infty {n\choose m}\frac{(1-aq^{2n})a^n q^{n^2}}{(1-q^n)(1-aq^n)}
+\sum_{k=1}^m \sum_{n=1}^\infty {n\choose m-k}\frac{a^n q^{n(n+k)}}{(1-q^n)^{k+1}}. \label{eq:main3}
\end{align}
\end{prop}
\begin{proof}
Similarly to the proof of Theorem~\ref{thm:main}, applying the Chu-Vandermonde convolution formula
$$
{n+j\choose m}=\sum_{k=0}^m {n\choose m-k}{j\choose k},
$$
and the binomial theorem
\begin{align*}
\sum_{j=k}^\infty {j\choose k} q^j
=\frac{q^k}{(1-q)^{k+1}},
\end{align*}
we can prove \eqref{eq:main3}.
\end{proof}

\subsection{A refinement of Uchimura's variation}
In fact, Liu's formula \eqref{eq:liu} is a special
case of Agarwal's identity, see \cite{Agarwal, Arndt}):
\begin{align*}
\sum_{n=0}^\infty \frac{t^n}{1-x q^n}
=\sum_{n=0}^\infty \frac{(1-xtq^{2n})x^n t^n q^{n^2}}{(1-xq^n)(1-tq^n)}.
\end{align*}
Another variation of \eqref{eq:liu} when $a=1$ is
Uchimura's identity~\cite{Uchimura84}:
\begin{align}
\sum_{k\geqslant 1}\frac{q^k}{1-q^k}=\sum_{m=1}^{+\infty} mq^m \prod_{j=m+1}^{+\infty}(1-q^j).  \label{eq:vh}
\end{align}
For $n\in \N$ let $d(n)$ be the number of positive divisors of $n$. A partition of $n$ is a non decreasing sequence
of positive integers $\pi=(n_1, n_2, \ldots, n_l)$ such
that $n_1\geqslant n_2\geqslant \cdots \geqslant n_l>0$ and $n=n_1+\cdots +n_l$.
The number of parts $l$ is called the length of $\pi$. The largest part $n_1$ is denoted by $g(\pi)$
and the smallest part $n_l$ is demoted by $s(\pi)$. A partition into distinct parts is called odd
(resp. even) if its length is odd (resp. even).
Let $t(n)$  be the sum of smallest parts of odd partitions of $n$ minus the sum of smallest parts of even partitions of $n$.
Bressoud and Subbarao~\cite{BS84}  noticed that Uchimura's identity is equivalent to
\begin{align}
d(n)=t(n),\label{eq:bs}
\end{align}
and gave a combinatorial  proof of \eqref{eq:vh}. It is interesting to note that \eqref{eq:bs} was rediscovered by
Wang et al. \cite{WFF95}. We can give a further refinement of \eqref{eq:bs}.

Let $d(n, N)$ be the number of divisors $\leqslant N$ of $n$.
Let ${\mathcal P}(n,N)$ be the set of partitions $\pi$ of $n$ into distinct parts
such that $g(\pi)-s(\pi)\leqslant N-1$. Let
$t(n, N)$ be the sum of smallest parts of odd partitions of ${\mathcal P}(n,N)$
minus the sum of smallest parts of even partitions of ${\mathcal P}(n,N)$. We have the following refinement of
\eqref{eq:bs}.

\begin{thm}
For all natural numbers $n$ and $N$, there holds
\begin{align}\label{eq:gvh}
d(n, N)=t(n, N)-t(n-N, N).
\end{align}
\end{thm}

For example, when $n=9$ and $N=3$, we have
$$
{\mathcal P}(9,3)=\{(9), (5,4), (4,3,2)\};\quad   {\mathcal P}(6,3)=\{(6), (4,2), (3,2,1)\},
$$
and so $t(9,3)-t(6,3)=(9+2-4)-(6+1-2)=2$, which is the number of divisors of $9$ less than or equal to 3.

\begin{proof}
Rewrite \eqref{eq:euler2}  as follows:
$$
\frac{1}{(zq;q)_N}=\sum_{m=0}^\infty \frac{(q^{m+1};q)_{N-1}}{(q;q)_{N-1}}z^mq^m.
$$
Differentiating the above identity with respect to $z$, we obtain
$$
\frac{1}{(zq;q)_N}\sum_{k=1}^N \frac{q^k}{1-zq^k}=\sum_{m=0}^\infty \frac{m(q^{m+1};q)_{N-1}}{(q;q)_{N-1}}z^{m-1}q^m.
$$
Multiplying  the two sides by $(zq;q)_N$ and setting  $z=1$, we get
\begin{align}
\sum_{k= 1}^N\frac{q^k}{1-q^k}=\sum_{m=1}^\infty mq^m(q^{m+1};q)_{N-1}
-\sum_{m=1}^\infty mq^{m+N}(q^{m+1};q)_{N-1}. \label{eq:qNN}
\end{align}
The proof then follows from extracting the coefficients of $q^n$ in both sides of \eqref{eq:qNN}.
\end{proof}

\begin{remark}
It is possible to give a combinatorial proof of  \eqref{eq:gvh} by generalizing
Bressoud and Subarao's combinatorial proof of \eqref{eq:bs}.
\end{remark}
\section{Generalizations of \eqref{eq:uch}, \eqref{eq:dilch} and \eqref{eq:prodinger}}
\subsection{Two generalizations of Uchimura's identity}
In this section, we give two generalizations of Uchimura's identity \eqref{eq:uch}.
\begin{thm}\label{thm:uch1}
For $m,n\geqslant 0$ and $0\leqslant r\leqslant m$, there holds
\begin{align}
&\hskip -3mm  \sum_{k=1}^{m}\frac{(-1)^k q^{{k+1\choose 2}-rk}}{(q;q)_k (q;q)_{m-k}(xq^{k};q)_{n+1}}
-x^r\sum_{k=1}^{n}\frac{(-1)^k q^{{k+1\choose 2}+rk}}
{(q;q)_k (q;q)_{n-k}(xq^{k};q)_{m+1}}
\nonumber\\
&=\frac{1}{(q;q)_{m}(q;q)_{n}}
\left(-\frac{1-x^r}{1-x}+\sum_{k=1}^n \frac{q^k (q;q)_{k-1}}{(xq;q)_k}
-x^r \sum_{k=1}^m \frac{q^k (q;q)_{k-1}}{(xq;q)_k}\right). \label{uch:001}
\end{align}
\end{thm}
\pf When $p=q$, the identity \eqref{eq:new} may be rewritten as
\begin{align}
&\hskip -3mm \sum_{k=1}^{m}\frac{(-1)^k q^{{k+1\choose 2}-rk}}{(q;q)_k (q;q)_{m-k}(xq^{k};q)_{n+1}}
-x^r\sum_{k=1}^{n}\frac{(-1)^k q^{{k+1\choose 2}+rk}}
{(q;q)_k (q;q)_{n-k}(xq^{k};q)_{m+1}}  \nonumber  \\
&=\frac{x^r}{(q;q)_n (x;q)_{m+1}}-\frac{1}{(q;q)_m (x;q)_{n+1}}.  \label{eq:pf11}
\end{align}
On the other hand, since
\begin{align*}
\sum_{k=1}^n \frac{q^k (q;q)_{k-1}}{(xq;q)_k}
=\sum_{k=1}^n\frac{1}{1-x}\left( \frac{(q;q)_{k-1}}{(xq;q)_{k-1}}-\frac{(q;q)_{k}}{(xq;q)_k}\right)
=\frac{1}{1-x}\left(1-\frac{(q;q)_{n}}{(xq;q)_n}\right),
\end{align*}
we have
\begin{align}
\sum_{k=1}^n \frac{q^k (q;q)_{k-1}}{(xq;q)_k}
-x^r\sum_{k=1}^m \frac{q^k (q;q)_{k-1}}{(xq;q)_k}
&=\frac{1-x^r}{1-x}+\frac{1}{1-x}\left(\frac{x^r(q;q)_{m}}{(xq;q)_m}-\frac{(q;q)_{n}}{(xq;q)_n}\right) \nonumber \\
&=\frac{1-x^r}{1-x}+\frac{x^r(q;q)_{m}}{(x;q)_{m+1}}-\frac{(q;q)_{n}}{(x;q)_{n+1}}.  \label{eq:pf12}
\end{align}
Combining \eqref{eq:pf11} and \eqref{eq:pf12}, we complete the proof of Theorem \ref{thm:uch1}.
\qed

\begin{thm}\label{thm:uch2}
For $m,n\geqslant 0$, there holds
\begin{align}
&\hskip -3mm  \sum_{k=1}^{m}\frac{(-1)^k q^{nk+{k+1\choose 2}}}
{(q;q)_k (q;q)_{m-k}(xq^{k};q)_{n+1}}
-\sum_{k=1}^{n}\frac{(-1)^k q^{mk+{k+1\choose 2}}}
{(q;q)_k (q;q)_{n-k}(xq^{k};q)_{m+1}}   \nonumber\\
&=\frac{1}{(q;q)_{m}(q;q)_{n}}
\left(\sum_{k=1}^n \frac{q^k (q;q)_{k-1}}{(xq;q)_k}
-\sum_{k=1}^m \frac{q^k (q;q)_{k-1}}{(xq;q)_k} \right). \label{uch:002}
\end{align}
\end{thm}
\pf Letting $v\to\infty$ and $z\to 0$ in the following identity (see \cite{GZ,GZh})
\begin{equation*}
\sum_{k=0}^{n}\frac{(q/z;q)_{k}(vq^m;q)_{k}(z;q)_{n-k}(xz;q)_m}
{(q;q)_k (v;q)_k (q;q)_{n-k}(xq^k;q)_{m+1}}z^k
=\sum_{k=0}^{m}\frac{(q/z;q)_{k}(vq^n;q)_{k}(z;q)_{m-k}(xz;q)_n}
{(q;q)_k (v;q)_{k}(q;q)_{m-k}(xq^k;q)_{n+1}}z^k,
\end{equation*}
we have
\begin{align*}
&\hskip -3mm \sum_{k=1}^{m}\frac{(-1)^k q^{nk+{k+1\choose 2}}}
{(q;q)_k (q;q)_{m-k}(xq^{k};q)_{n+1}}
-\sum_{k=1}^{n}\frac{(-1)^k q^{mk+{k+1\choose 2}}}
{(q;q)_k (q;q)_{n-k}(xq^{k};q)_{m+1}}   \\
&=\frac{1}{(q;q)_n (x;q)_{m+1}}- \frac{1}{(q;q)_m (x;q)_{n+1}}.
\end{align*}
The proof of Theorem \ref{thm:uch2} then follows from the $r=0$ case of \eqref{eq:pf12}.
\qed

It is easy to see that, when $m=0$ and $x=q^m$, both \eqref{uch:001} (with $r=0$) and \eqref{uch:002}
reduce to \eqref{eq:uch}.

\subsection{A generalization of Prodinger's identity}
In this section, we give a generalization of Prodinger's identity \eqref{eq:prodinger}.
\begin{thm}For $n\geqslant 0$ and $0\leqslant m, r\leqslant n$, there holds
\begin{align}
\sum_{\substack{k=0\\ k\neq m}}^{n}(-1)^{k-1}{n\brack k}_q\frac{q^{{k+1\choose 2}-rk}}{1-q^{k-m}}
=(-1)^m q^{{m+1\choose 2}-rm}{n\brack m}_q\left(r+\sum_{\substack{k=0\\ k\neq m}}^{n} \frac{q^{k-m}}{1-q^{k-m}}\right).
\label{eq:Prodinger-new}
\end{align}
\end{thm}
\pf From the partial fraction  decomposition
\begin{align*}
\sum_{k=0}^n\frac{(-1)^k q^{{k+1\choose 2}-rk}}{(q;q)_{k}(q;q)_{n-k}(1-xq^k)}
=\frac{x^r}{(x;q)_{n+1}},
\end{align*}
 which can also be obtained by letting $n=0$ and then replacing $(m,p)$ by $(n,q)$ in \eqref{eq:new},
we derive that
\begin{align}
&\hskip -3mm \sum_{\substack{k=0\\ k\neq m}}^{n}\frac{(-1)^k q^{{k+1\choose 2}-rk}}{(q;q)_{k}(q;q)_{n-k}(1-xq^k)}
\nonumber \\
&=\frac{x^r}{(x;q)_{n+1}}-\frac{(-1)^m q^{{m+1\choose 2}-rm}}{(q;q)_{m}(q;q)_{n-m}(1-xq^m)} \nonumber\\
&=\frac{x^r(q;q)_m(q;q)_{n-m}-(-1)^m q^{{m+1\choose 2}-rm}(x;q)_{m}(xq^{m+1};q)_{n-m}}
{(1-xq^m)(x;q)_{m}(xq^{m+1};q)_{n-m}(q;q)_m(q;q)_{n-m}}. \label{eq:Prodinger-r}
\end{align}
Letting $x\to q^{-m}$ in \eqref{eq:Prodinger-r}  and applying l'H\^opital's rule,
we obtain
\begin{align}
&\hskip -3mm \sum_{\substack{k=0\\ k\neq m}}^{n}
\frac{(-1)^{k} q^{{k+1\choose 2}-rk}}{(q;q)_k (q;q)_{n-k}(1-q^{k-m})}  \nonumber\\
&=\frac{rq^{-rm}(q;q)_m(q;q)_{n-m}+(-1)^m q^{{m+1\choose 2}-rm}
(q^{-m};q)_{m}(q;q)_{n-m}\sum_{k=0, k\neq m}^{n} \frac{q^{k-m}}{1-q^{k-m}}}
{-(q^{-m};q)_{m}(q;q)_{n-m}(q;q)_m(q;q)_{n-m}},
\end{align}
which is equivalent to \eqref{eq:Prodinger-new}.
\qed

Letting $m=0$ in \eqref{eq:Prodinger-new}, we get the following generalization of Van Hamme's identity
\eqref{eq:hamme}.
\begin{cor}For $n\geqslant 0$ and $0\leqslant r\leqslant n$, there holds
\begin{align}
\sum_{k=1}^{n}(-1)^{k-1}{n\brack k}_q\frac{q^{{k+1\choose 2}-rk}}{1-q^k}
=r+\sum_{k=1}^n \frac{q^k}{1-q^k}. \label{eq:rdiv}
\end{align}
\end{cor}
Letting $n\to\infty$ in \eqref{eq:rdiv}, we obtain
\begin{align*}
\sum_{k=1}^{\infty}(-1)^{k-1}\frac{q^{{k+1\choose 2}-rk}}{(q;q)_k(1-q^k)}
=r+\sum_{k=1}^\infty \frac{q^k}{1-q^k},\quad r\geqslant 0,
\end{align*}
which is a generalization of Uchimura's identity \eqref{eq:u81}.

\subsection{A new generalization of Dilcher's identity}
Here we shall give a new generalization of Dilcher's identity \eqref{eq:dilch}.
\begin{thm}\label{thm:dilch-new}
For $m,n\geqslant 1$, and $0\leqslant r\leqslant m+n-1$, there holds
\begin{align}
&\hskip -3mm \sum_{k=1}^{n}(-1)^{k-1}{n\brack k}_q\frac{q^{{k\choose 2}+k(m-r)}}{(1-q^k)^m} \nonumber\\
&={r\choose m}+{r\choose m-1}\sum_{k_1=1}^n\frac{q^{k_1}}{1-q^{k_1}}
+{r\choose m-2}\sum_{1\leqslant k_1\leqslant k_2\leqslant n}\frac{q^{k_1+k_2}}{(1-q^{k_1})(1-q^{k_2})} \nonumber\\
&\quad{}+\cdots+{r\choose 0}\sum_{1\leqslant k_1\leqslant\cdots\leqslant k_m\leqslant n} \frac{q^{k_1+\cdots+k_m}}{(1-q^{k_1})\cdots(1-q^{k_m})}.
\label{eq:dilch-new}
\end{align}
\end{thm}
\pf
We proceed by induction on $m$ and on $n$. For $m=1$, the identity \eqref{eq:dilch-new} reduces to
\eqref{eq:rdiv}. Suppose that \eqref{eq:dilch-new} is true for some $m\geqslant 1$. We need to show that
it is also true for $m+1$, namely,
\begin{align}
&\hskip -3mm \sum_{k=1}^{n}(-1)^{k-1}{n\brack k}_q\frac{q^{{k\choose 2}+k(m+1-r)}}{(1-q^k)^{m+1}} \nonumber\\
&={r\choose m+1}+{r\choose m}\sum_{k_1=1}^n\frac{q^{k_1}}{1-q^{k_1}}
+{r\choose m-1}\sum_{1\leqslant k_1\leqslant k_2\leqslant n}\frac{q^{k_1+k_2}}{(1-q^{k_1})(1-q^{k_2})} \nonumber\\
&\quad{}+\cdots+{r\choose 0}\sum_{1\leqslant k_1\leqslant\cdots\leqslant k_{m+1}\leqslant n}
\frac{q^{k_1+\cdots+k_{m+1}}}{(1-q^{k_1})\cdots(1-q^{k_{m+1}})},\quad 0\leqslant r\leqslant m+n.
\label{eq:dilch-new2}
\end{align}
We shall prove this induction step \eqref{eq:dilch-new2} by induction on $n$, following the proofs
in \cite{Hoffman} and \cite[Theorem 4]{Dilcher}.
For $n=1$, the left-hand side of \eqref{eq:dilch-new2} is equal to
$$
\frac{q^{m+1-r}}{(1-q)^{m+1}},
$$
while the right-hand side of \eqref{eq:dilch-new2} is given by
$$
\sum_{k=0}^{m+1}{r\choose m+1-k}\left(\frac{q}{1-q}\right)^k
=\left(\frac{q}{1-q}\right)^{m+1-r}\left(1+\frac{q}{1-q}\right)^r
=\frac{q^{m+1-r}}{(1-q)^{m+1}},
$$
since $0\leqslant r\leqslant m+1$. This proves that \eqref{eq:dilch-new2} holds for $n=1$.
We now assume that \eqref{eq:dilch-new2} holds for $n-1$. In order to show that it also holds for $n$,
we need to check that the difference between \eqref{eq:dilch-new2} for $n$ and \eqref{eq:dilch-new2} for $n-1$
is a true identity. This difference is
\begin{align}
&\sum_{k=1}^{n}(-1)^{k-1}\frac{q^{{k\choose 2}+k(m+1-r)}}{(1-q^k)^{m+1}}
\left({n\brack k}_q-{n-1\brack k}_q\right) \nonumber\\
&={r\choose m}\frac{q^{n}}{1-q^{n}}
+{r\choose m-1}\frac{q^{n}}{1-q^{n}}\sum_{1\leqslant k_1\leqslant n}\frac{q^{k_1}}{1-q^{k_1}} \nonumber\\
&\quad{}+\cdots+{r\choose 0}\frac{q^{n}}{1-q^{n}}\sum_{1\leqslant k_1\leqslant\cdots\leqslant k_{m}\leqslant n}
\frac{q^{k_1+\cdots+k_{m}}}{(1-q^{k_1})\cdots(1-q^{k_{m}})}. \label{eq:dilch-new3}
\end{align}
Applying the relation
\begin{align*}
\frac{q^{{k\choose 2}+k(m+1-r)}}{1-q^k}\left({n\brack k}_q-{n-1\brack k}_q\right)
=\frac{q^{{k\choose 2}+k(m-r)+n}}{1-q^n}{n\brack k}_q
\end{align*}
one sees that \eqref{eq:dilch-new3} is just the induction hypothesis \eqref{eq:dilch-new}.
This proves that \eqref{eq:dilch-new2} holds for all $n\geqslant 1$,
and consequently \eqref{eq:dilch-new} holds for all $m\geqslant 1$
and all $n\geqslant 1$.  \qed

By the $q$-binomial theorem \eqref{eq:euler1}, we have
$$
\sum_{k=1}^{n}(-1)^{k-1}{n\brack k}_{q} q^{k\choose 2}=1\qquad (n\geqslant 1).
$$
Hence, by Dilcher's identity \eqref{eq:dilch}, from \eqref{eq:dilch-new} we deduce the following result.
\begin{cor}
For $m,n\geqslant 1$, and $0\leqslant r\leqslant m+n-1$, there holds
\begin{align*}
\sum_{k=1}^{n}(-1)^{k-1}{n\brack k}_q\frac{q^{{k\choose 2}+k(m-r)}}{(1-q^k)^m}
&=\sum_{j=0}^{m}{r\choose m-j}\sum_{k=1}^{n}(-1)^{k-1}{n\brack k}_q\frac{q^{{k\choose 2}+kj}}{(1-q^k)^j}.
\end{align*}
\end{cor}

\vskip 5mm
\noindent{\bf Acknowledgments.} This work was partially
supported by the Fundamental Research Funds for the Central Universities,
the National Natural Science Foundation of China (grant 11371144).

\end{document}